\documentclass[oneside,english]{amsart}
\usepackage[T1]{fontenc}
\usepackage[latin9]{inputenc}
\usepackage{amsthm}
\usepackage{amstext}
\usepackage{amssymb}
\usepackage{esint}

\makeatletter

\numberwithin{equation}{section}
\numberwithin{figure}{section}
\theoremstyle{plain}
\newtheorem{thm}{\protect\theoremname}[section]
  \theoremstyle{plain}
  
  \theoremstyle{plain}
  
  \theoremstyle{plain}
  
  \theoremstyle{plain}
  \newtheorem{lem}[thm]{\protect\lemmaname}
  \theoremstyle{definition}

\makeatother

\usepackage{babel}
  \providecommand{\corollaryname}{Corollary}
  \providecommand{\definitionname}{Definition}
  \providecommand{\lemmaname}{Lemma}
  \providecommand{\propositionname}{Proposition}
  \providecommand{\examplename}{Example}
\providecommand{\theoremname}{Theorem}

\DeclareMathOperator{\cp}{cap}
\DeclareMathOperator{\Lip}{Lip}

\DeclareMathOperator{\loc}{loc}

\DeclareMathOperator{\dist}{dist}
\DeclareMathOperator{\diam}{diam}
\DeclareMathOperator{\esssup}{esssup}

\DeclareMathOperator{\ACL}{ACL}

\begin{document}

\title{Sobolev mappings and moduli inequalities on  Carnot groups}

\author{E.~Sevost'yanov, A.Ukhlov}

\begin{abstract}
In the article we study mappings of Carnot groups satisfy moduli inequalities.  We prove that homeomorphisms satisfy the moduli inequalities ($Q$-homeomor\-phisms) with a locally integrable function $Q$ are Sobolev mappings. On this base in the frameworks of the weak inverse mapping theorem we prove that mappings inverse to Sobolev homeomorphisms of finite distortion of the class $W^1_{\nu,\loc}(\Omega;\Omega')$ belong to the Sobolev class $W^1_{1,\loc}(\Omega';\Omega)$. 
\end{abstract}

\maketitle


\section{Introduction }

It is known that Sobolev mappings on Carnot groups $\mathbb G$ can not be characterized in the terms of its coordinate functions. The basic approach to the Sobolev mappings theory on Carnot groups is based on the notion of absolutely continuity on almost all horizontal lines which allows to define a weak upper gradient of mappings. In the present article we prove that homeomorphisms satisfy moduli inequalities on Carnot groups are Sobolev mappings. On this base we prove the weak version of the inverse mapping theorem on Carnot groups. Namely we prove that mappings inverse to Sobolev homeomorphisms of finite distortion of the class $W^1_{\nu,\loc}(\Omega;\Omega')$ are Sobolev mappings of the class $W^1_{1,\loc}(\Omega';\Omega)$. The problem of regularity of mappings inverse to Sobolev homeomorphisms represents a significant part of the weak inverse mapping theorem and was studied in \cite{Z69} for a bi-measurable Sobolev homeomorphism $\varphi: \Omega\to\Omega'$, $\Omega,\Omega'\subset \mathbb R^n$ of the class $W^1_p(\Omega;\Omega')$, $p>n-1$. In \cite{U04} it was proved that the inverse of a homeomorphism $\varphi\in L^1_p(\Omega;\Omega')$, $p>n-1$, satisfies $\varphi^{-1}\in BV_{\loc}(\Omega';\Omega)$. In the last decades the regularity of mappings inverse to Sobolev homeomorphisms  was intensively studied in the frameworks of the non-linear elasticity theory \cite{B81}, see, for example, \cite{CHM10,GU10,HKM06,HKO07,O07}.

The suggested approach is based on the moduli inequalities, namely on the notion of $Q$-mappings introduced in \cite{MRSY} (see also \cite{MRSY$_1$}--\cite{MRSY$_2$}).
Recall that a homeomorphism $\varphi: \Omega\to\Omega'$ of domains
$\Omega,\Omega'\subset \mathbb G$ is called a $Q$-homeomorphism, with a non-negative measurable function $Q$, if
$$
M\left(\varphi \Gamma\right)\leqslant \int\limits_{\Omega} Q(x)\cdot
\rho^{\nu}(x)dx
$$
for every family $\Gamma$ of rectifiable paths in $\Omega$ and every admissible function $\rho$ for $\Gamma$.

In the Euclidean space $\mathbb R^n$  it was proved \cite{MRSY$_1$} that a homeomorphism $\varphi\in W^{1}_{n,\loc}(\Omega)$ such that $\varphi^{\,-1}\in W^{1}_{n,\loc}$ is a $Q$-mapping with $Q=K_I(x, \varphi)$,
where $K_I(x, \varphi)$ is the inner dilatation of $\varphi$. The systematic applications of the moduli theory to the geometric mapping theory can be found in \cite{MRSY09}.

The main result of the article concerns to the weak differentiability of mappings satisfy moduli inequalities on Carnot groups (Theorem~\ref{ThmACL}). The proof is based on the capacity estimates and the Fubini type decomposition of measures associated with horizontal foliations defined by a left-invariant vector fields and moduli (capacity) inequalities on Carnot groups.

Using the property of the weak differentiability and connection between Sobolev mappings and moduli inequalities we prove the weak regularity of Sobolev homeomorphisms on Carnot groups: if $\varphi: \Omega\to\Omega'$ is a Sobolev homeomorphism of finite distortion of the class $W^1_{\nu,\loc}(\Omega;\Omega')$, then the inverse mapping $\varphi^{-1}\in W^1_{1,\loc}(\Omega';\Omega)$.

The weak differentiability is a part of the analytic definition of quasiconformal mappings and mappings of bounded distortion (see, e.g., \cite{Re} and \cite{MRV$_1$}). The $\ACL$-property of $Q$-mappings defined on planar domains of the Euclidean space $\mathbb R^2$ was considered by Brakalova and Jenkins, who proved this property for solutions of Beltrami equations in the plane (see \cite[Lemma~3]{BJ}). Under the assumption that $Q\in L_{\loc}^1,$ the $\ACL$-property was proved in ${\mathbb R}^n$ for $Q$-homeomorphisms (see \cite{Sal}), and for mappings with branching later (see e.g. \cite{SalSev$_1$}, \cite{SalSev$_2$}).

$Q$-homeo\-mor\-phisms are closely connected with mappings that generate bounded composition operators on Sobolev spaces ($p,q$-quasicon\-for\-mal mappings) \cite{GGR95,U93,VU98,VU02} which were studied on Carnot groups in \cite{U11,UV10,VU98,VU04}.
In the recent decade the geometric theory of composition operators on Sobolev spaces was applied to spectral estimates of the Laplace operator in Euclidean non-convex domains (see, for example, \cite{BGU1,BGU2,GHU18,GPU18_3,GU16,GU17}) and so results of this article have applications to the Sobolev mappings theory, to the spectral theory of (sub)elliptic operators and to the non-linear elasticity problems associated with vector fields that satisfy H\"ormander's hypoellipticity condition.

\section{Sobolev mappings on Carnot groups}

\subsection{Carnot groups}

Recall that a stratified homogeneous group \cite{FS}, or, in another
terminology, a Carnot group \cite{Pa} is a~connected simply
connected nilpotent Lie group~ $\mathbb G$ whose Lie algebra~ $V$ is
decomposed into the direct sum~ $V_1\oplus\cdots\oplus V_m$ of
vector spaces such that $\dim V_1\geqslant 2$, $[V_1,\ V_i]=V_{i+1}$
for $1\leqslant i\leqslant m-1$ and $[V_1,\ V_m]=\{0\}$. Let
$X_{11},\dots,X_{1n_1}$ be left-invariant basis vector fields of
$V_1$. Since they generate $V$, for each $i$, $1<i\leqslant m$, one
can choose a basis $X_{ik}$ in $V_i$, $1\leqslant k\leqslant
n_i=\dim V_i$, consisting of commutators of order $i-1$ of fields
$X_{1k}\in V_1$. We identify elements $g$ of $\mathbb G$ with
vectors $x\in\mathbb R^N$, $N=\sum_{i=1}^{m}n_i$, $x=(x_{ik})$,
$1\leqslant i\leqslant m$, $1\leqslant k\leqslant n_i$ by means of
exponential map $\exp(\sum x_{ik}X_{ik})=g$. Dilations $\delta_t$
defined by the formula 
\begin{multline}
\nonumber
\delta_t x= (t^ix_{ik})_{1\leqslant i\leqslant m,\,1\leqslant k\leqslant n_j}\\
=(tx_{11},...,tx_{1n_1},t^2x_{21},...,t^2x_{2n_2},...,t^mx_{m1},...,t^mx_{mn_m}),
\end{multline}
are automorphisms of
$\mathbb G$ for each $t>0$. Lebesgue measure $dx$ on $\mathbb R^N$
is the bi-invariant Haar measure on~ $\mathbb G$ (which is generated
by the Lebesgue measure by means of the exponential map), and
$d(\delta_t x)=t^{\nu}~dx$, where the number
$\nu=\sum_{i=1}^{m}in_i$ is called the homogeneous dimension of the
group~$\mathbb G$. The measure $|E|$ of a measurable subset
$E$ of $\mathbb G$ is defined by
$$
|E|=\int\limits_E~dx.
$$

The system of basis vectors $X_1,X_2,\dots,X_n$ of the space
$V_1$ (here and throughout we set $n_1=n$ and $X_{i1}=X_i$, where
$i=1,\dots,n$) satisfies the H\"ormander's hypoellipticity condition.

Euclidean space $\mathbb R^n$ with the standard structure is an
example of an abelian group: the vector fields $\partial/\partial
x_i$, $i=1,\dots,n$, have no non-trivial commutation relations
and form the basis of the corresponding Lie algebra. One example of
a non-abelian stratified group is the Heisenberg group $\mathbb
H^n$.
The non-commutative multiplication is defined as
$$
hh'=(x,y,z)(x',y',z')=(x+x', y+y', z+z'-2xy'+2yx'),
$$
where $x,x',y,y'\in\mathbb R^n$, $z,z'\in\mathbb R$.
Left translation $L_h(\cdot)$ is defined as $L_h(h')=hh'$.
The left-invariant vector fields
$$
X_i=\frac{\partial}{\partial x_i}+2y_i\frac{\partial}{\partial z},\,\,Y_i=\frac{\partial}{\partial y_i}-2x_i\frac{\partial}{\partial z},\,\,i=1,...,n,\,\,Z=\frac{\partial}{\partial z},
$$
constitute the basis of the Lie algebra $V$ of the Heisenberg group $\mathbb H^n$. All non-trivial relations are only of the form
$\left[X_i,Y_i\right]=-4Z$, $i=1,...,n$, and all other commutators vanish.

The Lie algebra of the Heisenberg group $\mathbb H^n$ has dimension $2n+1$
and splits into the direct sum $V=V_1\oplus V_2$. The vector space $V_1$ is generated by the vector fields $X_i, Y_i$, $i=1,...n$, and the space $V_2$ is the one-dimensional center which is spanned by the vector field $Z$.

Recall that a homogeneous norm on the group $\mathbb G$ is a continuous function
$|\cdot|:\mathbb G\to [0,\infty)$ that is $C^{\infty}$-smooth on $\mathbb
G\setminus\{0\}$ and has the following properties:

(a) $|x|=|x^{-1}|$ and $|\delta_t(x)|=t|x|$;

(b) $|x|=0$ if and only if $x=0$;

(c) there exists a constant $\tau_0>0$ such that $|x_1
x_2|\leqslant \tau_0(|x_1|+|x_2|)$ for all $x_1,x_2\in \mathbb G$.

The homogeneous norm on the group $\mathbb G$ define a homogeneous (quasi)metric
$$
\rho(x,y)=|y^{-1} x|.
$$

Note that a continuous map $\gamma: [a,b]\to\mathbb G$ is called a continuous curve on $\mathbb G$. This continuous curve is rectifiable if
$$
\sup\left\{\sum\limits_{k=1}^m|\left(\gamma(t_{k})\right)^{-1}\gamma(t_{k+1})|\right\}<\infty,
$$
where the supremum is taken over all partitions $a=t_1<t_2<...<t_m=b$ of the segment $[a,b]$.

In \cite{Pa} it was proved that any rectifiable curve is differentiable almost everywhere and $\dot{\gamma}(t)\in V_1$: there exists measurable functions $a_i(t)$, $t\in (a,b)$ such that
$$
\dot{\gamma}(t)=\sum\limits_{i=1}^n a_i(t)X_i(\gamma(t))\,\,\text{and}\,\,
\left|\left(\gamma(t+\tau)\right)^{-1}\gamma(t)exp(\dot{\gamma}(t)\tau)\right|=o(\tau)\,\,\text{as}\,\,\tau\to 0
$$
for almost all $t\in (a,b)$.
The length $l(\gamma)$ of a rectifiable curve $\gamma:[a,b]\to\mathbb G$ can be calculated by the formula
$$
l(\gamma)=\int\limits_a^b {\left\langle \dot{\gamma}(t),\dot{\gamma}(t)\right\rangle}_0^{\frac{1}{2}}~dt=
\int\limits_a^b \left(\sum\limits_{i=1}^{n}|a_i(t)|^2\right)^{\frac{1}{2}}~dt
$$
where ${\left\langle \cdot,\cdot\right\rangle}_0$ is the inner product on $V_1$. The result of \cite{CH} implies that one can connect two arbitrary points $x,y\in \mathbb G$ by a rectifiable curve. The Carnot-Carath\'eodory distance $d(x,y)$ is the infimum of the lengths over all rectifiable curves with endpoints $x$ and $y$ in $\mathbb G$. The Hausdorff dimension of the metric space $\left(\mathbb G,d\right)$ coincides with the homogeneous dimension $\nu$ of the group $\mathbb G$.

\subsection{Sobolev spaces on Carnot groups}

Let $\mathbb G$ be a Carnot group with one-parameter dilatation
group $\delta_t$, $t>0$, and a homogeneous norm $\rho$, and let
$E$ be a measurable subset of $\mathbb G$. The Lebesgue space $L_p(E)$, $p\in [1,\infty]$, is the space of pth-power
integrable functions $f:E\to\mathbb R$ with the standard norm:
$$
\|f\mid
L_p(E)\|=\biggl(\int\limits_{E}|f(x)|^p~dx\biggr)^{\frac{1}{p}},\,\,1\leq p<\infty,
$$
and $\|f\mid
L_{\infty}(E)\|=\esssup_{E}|f(x)|$ for $p=\infty$. We
denote by $L_{p,\loc}(E)$ the space of functions
$f: E\to \mathbb R$ such that $f\in L_p(F)$ for each compact
subset $F$ of $E$.

Let $\Omega$ be an open set in $\mathbb G$. The (horizontal) Sobolev space
$W^1_p(\Omega)$, $1\leqslant p\leqslant\infty$, $(L^1_p(\Omega)$,
$1\leqslant p\leqslant\infty$) consists of the functions
$f:\Omega\to\mathbb R$ locally integrable in $\Omega$, having a weak
derivatives $X_i f$ along the horizontal vector fields $X_i$, $i=1,\dots,n$,
and a finite (semi)norm
$$
\|f\mid W^1_p(\Omega)\|=\|f\mid L_p(\Omega)\|+\|\nabla_H f\mid
L_p(\Omega)\|\,\,\,\,\,(\|f\mid L^1_p(\Omega)\|=\|\nabla_H f\mid
L_p(\Omega)\|),
$$
where $\nabla_H f=(X_1f,\dots,X_nf)$ is the horizontal subgradient of $f$.
If $f\in W^1_p(U)$ for each bounded open set $U$ such that
$\overline{U}\subset\Omega$ then we say that $f$ belongs to the
class $W^1_{p,\loc}(\Omega)$.

Let $\varphi : \Omega\to\mathbb G$ be a mapping defined on open set $\Omega\subset\mathbb G$. A Lie group homomorphism $\psi:\mathbb G\to\mathbb G$ such that $\exp^{-1}\circ\psi\circ \exp(V_1)\subset V_1$ is called the $P$-differential of $\varphi$ at the point $a$ of the set $\Omega$ if the set
$$
A_{\varepsilon} =\{z\in E:
d(\psi(a^{-l}x)^{-1}\varphi(a)^{-1}\varphi(x))<\varepsilon
d(a^{-1}x)\}
$$
is a neighborhood of $a$ (relative to $\Omega$) for every $\varepsilon>0$. The notion of $P$-differen\-tiabi\-lity was introduced in \cite{Pa}
where it was proved that Lipschitz mappings defined on open subsets of Carnot groups are $P$-differentiable almost everywhere. The Stepanov type theorem on Carnot groups was obtained in \cite{VU96} (see, also \cite{Vod4}) where it was proved that Lipschitz mappings defined on measurable subsets of Carnot groups are (approximately) $P$-differentiable almost everywhere.

We say that a mapping
$\varphi:\Omega\to\mathbb G$ is absolutely continuous on lines ($\varphi\in
\ACL(\Omega;\mathbb G)$) if for each domain $U$ such that
$\overline{U}\subset\Omega$ and each foliation $\Gamma_i$ defined
by a left-invariant vector field $X_i$, $i=1,\dots,n$, $\varphi$ is
absolutely continuous on $\gamma\cap U$ with respect to
one-dimensional Hausdorff measure for $d\gamma$-almost every curve
$\gamma\in\Gamma_i$. Recall that the measure $d\gamma$ on the
foliation $\Gamma_i$ equals  the inner product $i(X_i)dx$ of the
vector field $X_i$ and the bi-invariant volume $dx$ ( see, for
example, \cite{Fe, VU96}).

Since $X_i \varphi(x)\in{V}_1$ for almost all $x\in\Omega$ \cite{Pa},
$i=1,\dots,n$, the linear mapping $D_H \varphi(x)$ with matrix
$(X_i\varphi_j(x))$, $i,j=1,\ldots,n$, takes the horizontal subspace $V_1$ to $V_1$
and is called the formal
horizontal differential of the mapping $\varphi$ at $x$. Let $|D_H
\varphi(x)|$ be its norm:
$$
|D_H\varphi(x)|=\sup\limits_{\xi\in
V_1,\,|\xi|=1} |D_H\varphi(x)(\xi)|.
$$

We say that a mapping $\varphi:\Omega\to\mathbb G$ belongs to $\ACL_p(\Omega;\mathbb G)$) ($\ACL_{p,\loc}(\Omega;\mathbb G)$)) if
$\varphi\in\ACL(\Omega;\mathbb G)$ and $|D_H\varphi|\in L_p(\Omega)$ ($|D_H\varphi|\in L_{p,\loc}(\Omega)$).

Smooth mappings with differentials respecting the horizontal
structure are said to be contact. For this reason one could say that
mappings in the class $\ACL(\Omega;\mathbb G)$ are (weakly) contact.
It was proved in \cite{Vod4, VU96} that a formal horizontal
differential $D_H:V_1\to{V}_1$ induces a homomorphism
$D\varphi:V\to{V}$ of the Lie algebras which is called the formal
differential. The determinant of the matrix
$D\varphi(x)$ is called the (formal) Jacobian of the mapping $\varphi$, it is
denoted by $J(x,\varphi)$.

The definition of Sobolev mappings in terms of Lipschitz functions was introduced in \cite{U1, Vod4}:

Let $\Omega$ be a domain in a stratified group $\mathbb G$. The
mapping $\varphi:\Omega\to{\mathbb G}$ belongs to
$W^1_{p,\loc}(\Omega;{\mathbb G})$ if for each function
$f\in\Lip({\mathbb G})$ the composition $f\circ \varphi$ belongs to
$W^1_{p,\loc}(\Omega)$ and $|\nabla_H(f\circ \varphi)|(x)\leqslant \Lip
f\cdot g(x)$, where $g\in L_{p,\loc}(\Omega)$ is independent of $f$.
The function $g$ is called the upper gradient of the mapping
$\varphi$.

\section{Foliations and Set Functions}

\subsection{The Fubini type decomposition}

We consider families $\Gamma_k$ of orbits of horizontal vector
fields $X_{1k}\in V_1$, $1\leqslant k\leqslant n_1$, generating smooth
foliations of a domain $\Omega\subset {\mathbb G}$. Denote the flow
corresponding to the vector field $X_{1k}$ by the symbol $f_t$, then
each fiber has the form $\gamma(t)=f_t(s)$, where $s$ belongs to the
surface $S_k$ transversal to $X_{1k}$ and a parameter $t\in \mathbb
R$.

We suppose that the foliation $\Gamma_k$ of $\Omega$ is furnished with a measure $d\gamma$
satisfying the inequality

\begin{equation}
c_1 |B(x, r)|^{\frac{\nu-1}{\nu}}\leqslant \int\limits_{\gamma\in\Gamma, \gamma\cap B(x, r)\ne \emptyset} d\gamma\leqslant c_2 |B(x, r)|^{\frac{\nu-1}{\nu}}
\label{eq:Fubini}
\end{equation}
for sufficiently small balls $B(x, r)\subset \Omega$ where
constants $c_1$ and $c_2$ independent on balls $B(x, r)$.

The measure $d\gamma$ can be obtained \cite{VU96} as the interior multiplication
$i(X_{1k})$ of the vector field $X_{1k}$ with the bi-invariant volume form $dx$. Let $J_{f_t}$ be a Jacobian of the flow $f_t$. Then
$$
f_t^{\ast}i(X_{1k})dx=J_{f_t}i(X_{1k})dx\,\,\text{or}\,\, f_t^{\ast}\left(J_{f_{-t}}i(X_{1k})dx\right)=i(X_{1k})dx.
$$

The tangent vector to a one-parameter family of curves $\gamma_t$ passing through points $s\exp t X_{1k}$ can be identified with the tangent vector $X_{1k}$ at the point $s\in S$. The flow $f_t$ takes the vector $X_{1k}$ to $\left(f_t\right)_{\ast}X_{1k}$. Consequently, the form $J_{f_{-t}}i(X_{1k})~dx$  determines the measure $d\gamma$ on the foliation $\Gamma_k$.

Note, that by the inequality (\ref{eq:Fubini}) the measure $d\gamma$ is the locally doubling measure:
\begin{equation}
\int\limits_{\gamma\in\Gamma_k, \gamma\cap B(x, 2r)\ne \emptyset}
d\gamma\leqslant c_d \int\limits_{\gamma\in\Gamma_k, \gamma\cap B(x,
r)\ne \emptyset} d\gamma \label{eq:doubling}
\end{equation}
for sufficiently small balls $B=B(x, r)\subset {\Omega}$.

Because $X_{1k}$ is a left-invariant vector field the flow $f_t$ is the right translation on $\exp t X_{1k}$. Since $dx$ is a bi-invariant form, we have $J_{f_{t}}=c_m$, where the constant $c_m$ can be calculated exactly. Using the left invariance and homogeneity under dilatations, we obtain that
\begin{equation}
\int\limits_{\gamma\in\Gamma_k, \gamma\cap B(x, r)\ne \emptyset} d\gamma=c_m |B(x, r)|^{\frac{\nu-1}{\nu}} \|X_{1k}\|
\label{eq:hom}
\end{equation}
where $\|X_{1k}\|$ is the length of the tangent vector $X_{1k}$.

\subsection{Additive set functions}

Recall that a mapping $\Phi$ defined on open subsets
from $\Omega\subset\mathbb G$ and taking nonnegative values is called a {\it finitely
quasiadditive} set function \cite{VU04} if

1) for any point $x\in \Omega$, exists $\delta$,
$0<\delta<\dist(x,\partial \Omega)$, such that $0\leqslant
\Phi(B(x,\delta))<\infty$ (here and in what follows
$B(x,\delta)=\{y\in\mathbb G: \rho(x,y)<\delta\}$);

2) for any  finite collection $U_i\subset U\subset \Omega$,
$i=1,\dots,k$, of mutually disjoint open sets the following
inequality $\sum\limits_{i=1}^k \Phi(U_i)\leqslant \Phi(U)$ takes
place.

Obviously, the inequality in the second condition of this
definition can be extended to a countable collection of mutually
disjoint open sets from $\Omega$, so a finitely quasiadditive set
function is also {\it countable quasiadditive.}

If instead of the second condition we suppose that for any finite
collection $U_i\subset \Omega$, $i=1,\dots,k$, of mutually disjoint
open sets the equality
$$
\sum\limits_{i=1}^k \Phi(U_i)= \Phi(U)
$$
takes place, then such a function is said to be {\it finitely
additive}. If the equality in this condition can be extended to a
countable collection of mutually disjoint open sets from $\Omega$,
then such a function is said to be {\it countably additive.}

A mapping $\Phi$ defined on open subsets of $\Omega$ and taking
nonnegative values is called a {\it monotone} set function
\cite{VU04} if $\Phi(U_1)\leqslant\Phi(U_2)$ under the condition that
$U_1\subset U_2\subset \Omega$ are open sets.

Let us formulate a result from \cite{VU04} in a form convenient for us.

\begin{thm} \cite{VU04} Let a finitely quasiadditive set
function $\Phi$ be defined on open subsets of the domain
$\Omega\subset\mathbb G$. Then for almost all points $x\in \Omega$ the
finite derivative
$$
\Phi'(x)=\lim\limits_{\delta\to 0, B_{\delta}\ni x}
\frac{\Phi(B_{\delta})}{|B_{\delta}|}
$$
exists and for any open set $U\subset \Omega$, the inequality
$$
\int\limits_{U}\Phi'(x)~dx\leqslant \Phi(U)
$$
holds.
\end{thm}

We consider the cube $P=S_k\exp t X_{1k}$, where $|t|\leqslant M$
and $S_k$ is the transversal hyperplane to $X_{1k}$:
$$
S_k = \left\{(x_{ij}), 1\leqslant i\leqslant m,1\leqslant j\leqslant
n_i: x_{1k}=0\,\,\text{and}\,\,|x_{ij}|\leqslant M\right\}.
$$

Given a point $s\in S_k$, denote by $\gamma_s$ the element $s\exp t X_{1k}$
of the horizontal fibration which starts at the point $s$. Thus $P$ is the union of all
such intervals of integral lines. Consider the following tubular
neighborhood of the fiber $\gamma_s$ with radius $r$:
$$
E(s, r) = \gamma_s B(e, r) \cap P=\left(\bigcup\limits_{x\in \gamma_s}B(x,r)\right)\cap P.
$$

The following lemma is valid (see \cite{VG95}):

\medskip
\begin{lem}\label{lem1}
Let $\Phi$ be a quasiadditive set function on ${\mathbb G}.$ Then
$$
\varlimsup\limits_{r\rightarrow 0}
\frac{\Phi(E(s,r))}{r^{\nu-1}} < \infty
$$ 
for $d\gamma$-almost all $s\in S_k$.
\end{lem}

\section{Capacity and Modules}

\subsection{The basic definitions}

A well-ordered triple $(F_0,F_1;\Omega)$ of nonempty sets, where $\Omega$ is
an open set in $\mathbb G$, and $F_0$, $F_1$ are compact subsets of
$\overline{\Omega}$, is called a condenser in the group $\mathbb G$.

The value
$$
\cp_p(F_0,F_1;\Omega)=\inf\int\limits_{\Omega} |\nabla_Hv|^p~dx,
$$
where the infimum is taken over all nonnegative functions $v\in
C(\Omega)\cap L^1_p(\Omega)$, such that $v=0$ in a
neighborhood of the set $F_0$, and $v\geqslant 1$ in a neighborhood
of the set $F_1$, is called the $p$-capacity of the condenser
$(F_0,F_1;\Omega)$. If $G\subset\mathbb G$ is an open set, and $E$ is a
compact subset in $G$, then the condenser $(\partial G, E; \mathbb
G)$  will be denoted by $(E,G)$. Properties of $p$-capacity in the
geometry of vector fields satisfying H\"ormander hypoellipticity
condition,  can be found in \cite{VCh1, VCh2}.

The linear integral is denoted by
$$
\int\limits_{\gamma}\rho~ds=\sup\int\limits_{\gamma'}\rho~ds=\sup\int\limits_0^{l(\gamma')}\rho(\gamma'(s))~ds
$$
where the supremum is taken over all closed parts $\gamma'$ of $\gamma$ and $l(\gamma')$ is the length of $\gamma'$. Let $\Gamma$ be a family of curves in $\mathbb G$. Denote by $adm(\Gamma)$ the set of Borel functions (admissible functions)
$\rho: \mathbb G\to[0,\infty]$ such that the inequality
$$
\int\limits_{\gamma}\rho~ds\geqslant 1
$$
holds for locally rectifiable curves $\gamma\in\Gamma$.

Let $\Gamma$ be a family of curves in $\overline{\mathbb G}$, where $\overline{\mathbb G}$ is a one point compactification of a Carnot group $\mathbb G$. The quantity
$$
M(\Gamma)=\inf\int\limits_{\mathbb G}\rho^{\nu}~dx
$$
is called the module of the family of curves $\Gamma$ \cite{M1}. The infimum is taken over all admissible functions
$\rho\in adm(\Gamma)$.

Let $\Omega$ be a bounded domain on $\mathbb G$ and $F_0, F_1$ be disjoint non-empty compact sets in the
closure of $\Omega$. Let $M(\Gamma(F_0,F_1;\Omega))$ stand for the
module of a family of curves which connect $F_0$ and $F_1$ in $\Omega$. Then \cite{M2}
\begin{equation}\label{eq2}
M(\Gamma(F_0,F_1;\Omega)) = \cp_{\nu}(F_0,F_1;\Omega)\,.
\end{equation}

\subsection{The lower estimate of the $p$-capacity}

The following lower estimate of the $p$-capacity was proved in
\cite[Lemma~5]{VU98}. For readers convenience we reproduce here the
detailed proof of this lemma.

\begin{lem}\label{lem2}
Let $\nu-1<p<\infty$. Suppose that $E$ is a compact connected set
and $G\subset\{x\in\mathbb G: \rho(x,E)\leqslant {c_0} \diam
E\}$, where $c_0$ is a small number depending on the constant in the generalized triangle
inequality. Then
\begin{equation}\label{eq1}
\cp_p^{\nu-1}(E,G)\geqslant c(\nu,p)\frac{(\diam E)^p}{|G|^{p-(\nu-1)}},
\end{equation}
where a constant $c(\nu,p)$ depends only on $\nu$ and
$p$.
\end{lem}

\begin{proof} Since the inequality (\ref{eq1}) is invariant under left translations and
has the same degree of homogeneity under dilations, we can suppose, without loss of generality, that $0\in E$ and
${\rm diam\, E} = \rho(0, \sigma) =1$ for some point $\sigma\in E$.

Consider a point $\sigma^{-1}\in S(0,1)$. Then there exists a constant $c_1$ such that
$$
\diam E=1\leqslant c_1(r_2-r_1),
$$
where $r_1=|\sigma^{-1}|=1$ and $r_2=\rho(\sigma^{-1},\sigma)=|\sigma^2|$.

Since $c_0$ was choosing such that $G\subset\{x\in\mathbb G: \rho(x,E)\leqslant c_0 \diam E\}$,
then by the generalized triangle inequality
$$
S(\sigma^{-1},r)\cap (\mathbb G\setminus
G)\ne\emptyset\,\,\,\text{for all}\,\,\,r_1\leqslant r\leqslant r_2.
$$

Let $r_1\leqslant r\leqslant r_2$. We choose some point $x_r\in E$
such that $\rho(\sigma^{-1},x_r)=r$ and denote
$$
P(r)=\left\{s\in S(\sigma^{-1},r): \rho(x_r,s)\leqslant \rho(x_r,
\{(\mathbb G\setminus G)\cap S(\sigma^{-1},r)\}\right\}.
$$

Consider an arbitrary function $u\in \mathring{L}_p^1(G)\cap C(G)$ such that
$u\geqslant 1$ on $E$. Then the function $u$ takes the value $0$ on the sphere $S(x, r),$
$r_1 < r < r_2.$ Therefore, the following inequality is valid for
almost all $r_1 < r < r_2$ \cite[Theorem~1]{Vod1}
$$\int\limits_{S(x, r)\cap G} M_{\gamma r}(|\nabla_H u|)^p(\xi)d\sigma_r(\xi)\geqslant c_2\omega_r
(P(r))^{\frac{\nu-1-p}{\nu-1}}\,,$$
where $\omega_r$ is the measure on $S(x, r)$ associated with the
"spherical" coordinate system \cite{Vod1}. (Here $\gamma > 1$ is
some constant and $M_{\delta}g$ denotes the maximal function defined
for every locally summable function $g$ as
$$M_{\delta}g(x)=\sup\left\{|B(x, r)|^{\,-1}\int\limits_{B(x, r)}|g|dx: r\leqslant \delta\right\}\,,$$
where $B(x,r) = \{y\in {\mathbb G} : \rho(x,y) < r\}$ is the
ball of radius $r$ centered at $x\in {\mathbb G}.$)
Consequently,
$$\int\limits_{G} M_{\gamma r}(|\nabla_H u|)^p dx\geqslant c_2\int\limits_{r_1}^{r_2}\omega_r
(P(r))^{\frac{\nu-1-p}{\nu-1}}dr\,.$$
Now
\begin{multline*}
({\rm diam}\,E)^p\leqslant \left(c_1\int\limits_{r_1}^{r_2}dr\right)^p\\
\leqslant
c_1^p\left(\int\limits_{r_1}^{r_2}\omega_r(P(r))dr\right)^{p-(\nu-1)}
\left(\int\limits_{r_1}^{r_2}\omega_r^{\frac{\nu-1-p}{\nu-1}}(P(r))dr\right)^{\nu-1}
\\
\leqslant \frac{c_1^p}{c_2}|G|^{p-(\nu-1)}\left(
\int\limits_{G} M_{\gamma r}(|\nabla_H u|)^p dx\right)^{\nu-1}\,.
\end{multline*}
By the maximal function theorem, we obtain
$$\left(
\int\limits_{G} |\nabla_H u|^p dx\right)^{\nu-1}\geqslant
c(\nu,p)\frac{({\rm diam}\,E)^p}{|G|^{p-(\nu-1)}}$$
for arbitrary function $u\in \mathring{L}_p^1(G)\cap C(G)$ admissible for the condenser $(E,G)$. Hence
$$
\cp_p^{\nu-1}(E,G)\geqslant c(\nu,p)\frac{(\diam E)^p}{|G|^{p-(\nu-1)}}.
$$
\end{proof}

\section{Sobolev spaces and $Q$-Homeomorphisms}

In this section we consider connection between Sobolev mappings and  $Q$-homeo\-mor\-phisms on Carnot groups. 

\subsection{$\ACL$-property of $Q$-homeomorphisms}

We prove the $\ACL$-property of $Q$-homeomorphisms with locally integrable function $Q$.

\begin{thm}
\label{ThmACL}
Let $\varphi: \Omega\to\Omega'$ be a $Q$-homeomorphism of domains  $\Omega,\Omega'\subset\mathbb G$ with $Q\in
L_{1,\loc}(\Omega)$. Then $\varphi\in W^1_{1,\loc}(\Omega;\Omega')$.
\end{thm}

\begin{proof}
Fix some field $X_{1k}$, $\leqslant k\leqslant n_1$, and let
$\Gamma_k$ be the fibration generated by this field. Take the cube
$P=S_k\exp t X_{1k}$, where $|t|\leqslant M$ and $S_k$ is the
transversal hyperplane to $X_{1k}$:
$$
S_k = \left\{(x_{ij}), 1\leqslant i\leqslant m,1\leqslant j\leqslant
n_i: x_{1k}=0\,\,\text{and}\,\,|x_{ij}|\leqslant M\right\}.
$$

Given a point $s\in S_k$, denote by $\gamma_s$ the element $s\exp t X_{1,k}$
of the fibration which starts at $s$. Thus, $P$ is the union of all
such intervals of integral lines. Consider the following tubular
neighborhood of the fiber $\gamma_s$ with radius $r$:
$$
E(s, r) = \gamma_s B(e, r) \cap P=\left(\bigcup\limits_{x\in \gamma_s}B(x,r)\right)\cap P.
$$

Take a point $s\in S_k$ so that the assertion of Lemma \ref{lem1} holds
for $\gamma_s$. On $\gamma_s$ take arbitrary pairwise disjoint
closed segments $\gamma_{s1},\ldots, \gamma_{sk}$ of lengths
$\delta_1,\ldots,\delta_k$. Denoting by $R_i$ the open set of points at a
distance less than a given $r>0$ from $\gamma_{si}$, $i=1,...,k$, and consider the
condensers $(\gamma_{si}, R_i)$, $i=1,...,k$. Suppose that $r> 0$ is chosen so small that the sets $R_1,\ldots, R_k$ are pairwise disjoint and
the condenser $(\varphi(\gamma_{si}), \varphi(R_i))$ satisfies to the conditions of Lemma \ref{lem2}. Let $\Gamma$ be a family of curves connected
$\varphi(\gamma_{si})$ and $\partial \varphi(R_i)$ in $\Omega.$ Now, by (\ref{eq2})
\begin{equation}\label{eq3}
M(\varphi(\Gamma)) = \cp_{\nu}(\varphi(\gamma_{si})),\varphi(R_i)) \,.
\end{equation}
Observe that the function
$$ \rho(x)\,=\,\left
\{\begin{array}{rr} \frac{1}{r}, & x\in R_i, \\
0, & x \in {\mathbb G}\setminus R_i
\end{array}\right.
$$
is admissible for $\Gamma$. Now by (\ref{eq3})
\begin{equation}\label{eq4}
\cp_{\nu}(\varphi(\gamma_{si})),\varphi(R_i))\leqslant
\frac{1}{r^{\nu}}\int\limits_{R_i} Q(x)\, dx\,.
\end{equation}

\medskip
On the other hand, by Lemma \ref{lem2}
\begin{equation}\label{2.5}
\cp_{\nu}(\varphi(\gamma_{si})),\varphi(R_i))\geqslant
c\left(\frac{\left(\diam\varphi(\gamma_{si})\right)^{\nu}}{|\varphi(R_i)|}\right)^{1/(\nu-1)}.
\end{equation}
\medskip

Combining (\ref{eq4}) and (\ref{2.5}), we have the inequalities
\begin{equation}\label{2.6}
\left(\frac{\left(\diam\varphi(\gamma_{si})\right)^{\nu}}{|\varphi(R_i)|}\right)^{\frac{1}{\nu-1}}\leqslant
\frac{c_\nu}{r^\nu}\int\limits_{R_i}Q(x)\ dx\ ,\ \ \ \ \ i=1,...,k
\end{equation} where the constant  $c_\nu$  depends only on $\nu.$
\medskip

By the discrete H\"older inequality, see e.g. (17.3) in \cite{BB}, we obtain that
\begin{equation}\label{2.7}
\sum\limits_{i=1}^{k}
\diam\varphi(\gamma_{si})\leqslant\left(\sum\limits_{i=1}^{k}\left(
\frac{\left(\diam\varphi(\gamma_{si})\right)^{\nu}}{|\varphi(R_i)|}\right)^{\frac{1}{\nu-1}}
\right)^{\frac{\nu-1}{\nu}}\left(\sum\limits_{i=1}^{k}
|\varphi(R_i)|\right)^{\frac{1}{\nu}}\,,
\end{equation}
i.e.,
\begin{equation}\label{2.8}
\left(\sum\limits_{i=1}^{k}
\diam\varphi(\gamma_{si})\right)^\nu\leqslant\left(\sum\limits_{i=1}^{k}\left(
\frac{\left(\diam\varphi(\gamma_{si})\right)^{\nu}}{|\varphi(R_i)|}\right)^{\frac{1}{\nu-1}}
\right)^{\nu-1}|\varphi(E(s, r))|\,,
\end{equation}
and in view of (\ref{2.6})
\begin{equation}\label{2.9}
\left(\sum\limits_{i=1}^{k}
\diam\varphi(\gamma_{si})\right)^\nu\leqslant c_\nu\,
\frac{|\varphi(E(s, r))|}{r^{\nu-1}}\
\left(\sum\limits_{i=1}^{k}\frac{\int\limits_{R_i}Q(x)\,dx}
{r^{\nu-1}}\right)^{\nu-1}\,
\end{equation} where a constant $c_\nu$ depends only on  $\nu$.

By \cite[Lemma~4]{VU98})
$$
\lim\limits_{r\to 0}\frac{|\varphi(E(s, r))|}{r^{\nu-1}}:=\omega(s)<\infty.
$$
Denote
$$
\omega_i(s)=\int\limits_{\delta_i}Q(s,t)~dt,\,\,s\in S_k,
$$
and note that because $Q$ is locally integrable function then by Fubini theorem for any $\varepsilon>0$ there exists a number $\delta>0$ such that $\omega_i(s)<\varepsilon$ if $\delta_i<\delta$, $i=1,...,k$.

By Fubini type decomposition (\ref{eq:hom}) we have
$$
\lim\limits_{r\to
0}\frac{\int\limits_{R_i}Q(x)\,dx}{r^{\nu-1}}=
\frac{\omega_i(s)}{c_m\|X_{1k}\|}<\infty.
$$

Passing in (\ref{2.9}) while $r\to 0$ we get
\begin{equation}\label{2.10}
\left(\sum\limits_{i=1}^{k}
\diam\varphi(\gamma_{si})\right)^\nu\leqslant
\frac{c_{\nu}\omega(s)}{c_m\|X_{1k}\|}\left(\sum\limits_{i=1}^{k}\omega_i(s)
\right)^{\nu-1}\,.
\end{equation}

Hence, $\varphi$ is absolutely continuous on $\gamma\cap P$ with respect to one-dimensional Hausdorff measure for $d\gamma$-almost every curve $\gamma\in\Gamma_k$. Hence  $\varphi\in W^1_{1,\loc}(\Omega;\Omega')$.

\end{proof}

\subsection{Mappings of integrable distortion}

Let a homeomorphism $\varphi: \Omega\to\Omega'$ belongs to the Sobolev space $W^1_{1,\loc}(\Omega;\Omega')$. Recall that a weakly differentiable mapping $\varphi: \Omega\to\Omega'$ is called a mapping of finite distortion if $|D_H\varphi(x)|=0$ for almost all $x\in Z=\{x\in\Omega : J(x,\varphi)=0\}$. We say that a homeomorphism $\varphi: \Omega\to\Omega'$ has the Luzin $N$-property, if an image of a set of measure zero has measure zero.

The outer dilatation of the mapping of finite distortion $\varphi$ at $x$ is defined by
$$
K_O(x)=K_O(x,\varphi)=
\begin{cases}
\frac{|D_H \varphi(x)|^{\nu}}{J(x,\varphi)}, & \,\,\text{if}\,\, J(x,\varphi)\ne 0,\\
0, & \,\,\text{if}\,\, D_H \varphi(x)=0.
\end{cases}
$$

\begin{thm}
\label{ThmQN}
Let $\varphi: \Omega\to\Omega'$ be a homeomorphism of finite distortion of the Sobolev class $W^1_{\nu,\loc}(\Omega;\Omega')$. Then, for every family $\Gamma$ of rectifiable paths in $\Omega$ and every $\rho\in adm(\Gamma)$
$$
M\left(\varphi^{-1}\left(\Gamma\right)\right)\leqslant\int\limits_{\Omega}
K_O\left(\varphi^{-1}(y),\varphi\right)\rho^{\nu}(y)~dy,
$$
i.~e., $\varphi^{-1}$ is a $Q$-homeomorphism with $Q(y)=K_O\left(\varphi^{-1}(y),\varphi\right)\in L_{1,\loc}(\Omega')$.
\end{thm}

\begin{proof}
Let $F$ be a compact subdomain of $\Omega$, $F'=\varphi(F)$. Denote 
$$
Z=\{x\in \Omega: J(x,\varphi)=0\}.
$$
Because $\varphi$ is a mapping of finite distortion then $|D_H\varphi|=0$ a.~e. on $Z$ and $K_O(x,\varphi)$ is well defined for almost all $x\in\Omega$. Since $\varphi\in W^1_{\nu,\loc}(\Omega)$ then $\varphi$ possesses the Luzin $N$-property and the outer distortion 
$K_O(\varphi^{-1}(y),\varphi)$ be well defined for almost all $y\in\Omega'$.
Then
\begin{multline*}
\int\limits_{F'} K_O\left(\varphi^{-1}(y),\varphi\right)~dy=\int\limits_{F'\setminus\varphi(Z)} \frac{|D_H\varphi(\varphi^{-1}(y))|^{\nu}}{|J(\varphi^{-1}(y),\varphi)|}~dy+\int\limits_{F'\cap \varphi(Z)} ~dy\\
=\int\limits_{F\setminus Z} \frac{|D_H\varphi(x)|^{\nu}}{|J(x,\varphi)|}|J(x,\varphi)|~dx+\int\limits_{F\cap Z} |J(x,\varphi)|~dx\\
=\int\limits_{F\setminus Z} |D_H\varphi(x)|^{\nu}~dx+\int\limits_{F\cap Z} |J(x,\varphi)|~dx<\infty.
\end{multline*}

Since $\varphi:\Omega\to\Omega'$ belongs to $W^1_{\nu,\loc}(\Omega;\Omega')$ then $\varphi$ be a (weakly) contact mapping differentiable almost everywhere in $\Omega$ and absolutely continuous on almost all horizontal curves. By generalized Fuglede's theorem (see, \cite{M3, Sh00}), we have that if $\tilde{\Gamma}$ is the family of all paths $\gamma\in\varphi^{-1}(\Gamma)$ such that $\varphi$ is absolutely continuous on all closed subpaths of $\gamma$, then $M(\varphi^{-1}(\Gamma))=M(\tilde{\Gamma})$.

Hence, for given a function $\rho\in adm\Gamma$ we define
\begin{equation}
\begin{cases}
\widetilde{\rho}(x)=\rho(\varphi(x))|D_H\varphi(x)|\,\,&\text{if}\,\, x\in\Omega,\\
0 &\text{otherwise}.
\end{cases}
\end{equation}

Then, for almost all $\tilde{\gamma}\in\tilde{\Gamma}$
$$
\int\limits_{\tilde{\gamma}}\tilde{\rho}~ds\geqslant\int\limits_{\varphi\circ{\widetilde{\gamma}}}{\rho}~ds\geqslant
1
$$
and consequently $\tilde{\rho}\in adm\tilde{\Gamma}$.

Therefore, using the change of variable formula \cite{VU96} we obtain:
\begin{multline}
M(\varphi^{-1}(\Gamma))=M(\tilde{\Gamma})\leq\int\limits_{\Omega}{\tilde{\rho}}^{\nu}(x)~dx=
\int\limits_{\Omega}\rho^{\nu}(\varphi(x))|D_H\varphi(x)|^{\nu}~dx=\\
\int\limits_{\Omega\setminus Z}\rho^{\nu}(\varphi(x))|D_H\varphi(x)|^{\nu}~dx
=\int\limits_{\Omega\setminus Z}\rho^{\nu}(\varphi(x))\frac{|D_H\varphi(x)|^{\nu}}{|J(x,\varphi)|}|J(x,\varphi)|~dx\\
=\int\limits_{\Omega'\setminus \varphi(Z)}\rho^{\nu}(y)\frac{|D_H\varphi(\varphi^{-1}(y))|}{|J(\varphi^{-1}(y),\varphi)|}~dy
=\int\limits_{\Omega'}K_O\left(\varphi^{-1}(y),\varphi\right)\rho^{\nu}(y)~dy.
\end{multline}
Hence $\varphi^{-1}$ is a $Q$-homeomorphism with $Q(y)=K_O\left(\varphi^{-1}(y),\varphi\right)\in L_{1,\loc}(\Omega')$.

\end{proof}

\subsection{The weak inverse mapping theorem on Carnot groups}

In this section we prove that mappings inverse to Sobolev homeomorphisms of finite distortion of the class $W^1_{\nu,\loc}(\Omega;\Omega')$ are Sobolev mappings.

\begin{thm}
\label{inverse}
Let $\varphi: \Omega\to\Omega'$ be a Sobolev homeomorphism of finite distortion of the class $W^1_{\nu,\loc}(\Omega;\Omega')$. Then $\varphi^{-1}\in W^1_{1,\loc}(\Omega;\Omega')$. 
\end{thm}

\begin{proof}
By Theorem~\ref{ThmQN} we obtain that the inverse mapping $\varphi^{-1}:\Omega'\to\Omega$ be a $Q$-homeomorphism with $Q\in L_{1,\loc}(\Omega')$. Hence using Theorem~\ref{ThmACL} we conclude that the inverse mapping $\varphi^{-1}\in W^1_{1,\loc}(\Omega';\Omega)$.
\end{proof}

\vskip 0.3cm

Department of Mathematical Analysis, Zhytomyr Ivan Franko State
University, 40 Bol'shaya Berdichevskaya Str., Zhytomyr, 10008,
Ukraine

\emph{E-mail address:} \email{esevostyanov2009@gmail.com}

\vskip 0.3cm

Department of Mathematics, Ben-Gurion University of the Negev, P.O.Box 653, Beer Sheva, 8410501, Israel

\emph{E-mail address:} \email{ukhlov@math.bgu.ac.il}

\end{document}